\documentclass[12pt,preprint]{elsarticle}
\usepackage[pagewise]{lineno}
\usepackage{amssymb,amsmath,amsthm,bm}
\usepackage{amsfonts}
\usepackage{epsfig}
\usepackage{amsbsy}
\usepackage{color}

\usepackage{CJK,indentfirst,amsmath,amsfonts,amssymb,amsthm}
\usepackage{graphicx}
\graphicspath{{figures/}}

\newlength\imagewidth
 \newtheorem{lemma}{Lemma}[section]
 \newtheorem{theorem}{Theorem}[section]

\newtheorem{remark}{Remark}[section]

\usepackage{pifont}
\renewcommand{\pi}
{\mathord{\text{\Pifont{psy}p}}}               
\renewcommand{\emptyset}
{\mathord{\raisebox{-0.25ex}{$\text{\Pisymbol{psy}{198}}$}}}         

\usepackage{color}
\definecolor{dgreen}{rgb}{0,.6,0}
\usepackage[normalem]{ulem}

\biboptions{sort&compress}

\def\ps@pprintTitle{%
     \let\@oddhead\@empty
     \let\@evenhead\@empty

     \def\@oddfoot{}%

     \let\@evenfoot\@oddfoot}
\def\@seccntDot{.}
\allowdisplaybreaks

\begin{document}
\begin{CJK*}{GBK}{song}

\begin{frontmatter}
\title{Chern-Simons Higgs models for $p$-Laplacian on finite graphs: a topological degree approach}

\author{Chunlian Liu}\ead{clliu06@163.com}

\author{Yating Ge}\ead{2755524361@qq.com}

\author{Linfeng Wang\corref{cor1}}\cortext[cor1]{Corresponding author}
\ead{wlf711178@ntu.edu.cn}

\address{School of Mathematics and Statistics, Nantong University, Nantong 226019, China}

\begin{abstract}
We investigate the Chern-Simons Higgs models for $p$-Laplacian on a connected finite graph, employing topological degree theory as our main tool. Notably, we overcome the difficulties arising from the nonlinearity of $p$-Laplacian operator and calculate the corresponding topological degree through a more general approach.
\end{abstract}

\begin{keyword}
Topological degree; Chern-Simons Higgs models; finite graphs; $p$-Laplacian

\medskip
\MSC 53C21, 05C99

\end{keyword}

\end{frontmatter}

\section{Introduction}
\setcounter{equation}{0}

Let $G=(V,E)$ be a connected finite graph, where $V$ denotes the set of vertices and $E$ the set of edges. For $p>1$, the $p$-Laplacian of a function $u: V \rightarrow \mathbb{R}$ is defined as
\[\Delta_p u(x)=\sum_{y \sim x}|u(y)-u(x)|^{p-2}(u(y)-u(x)),\]
where $y\sim x$ indicates that $y$ is adjacent to $x$, i.e. $x y \in E$. In particular, the $2$-Laplacian of a function $u: V \rightarrow \mathbb{R}$ is defined as
\[\Delta u(x)=\sum_{y \sim x}(u(y)-u(x)).\]

We are interested in the Chern-Simons Higgs models for $p$-Laplacian of the form
\begin{equation}\label{1.1}
\Delta_pu=\lambda e^u(e^u-1)+f
\end{equation}
on $G$, where $\lambda\in\mathbb{R}$, and $f:V\to\mathbb{R}$ is a given function.

The study of Chern-Simons Higgs models can be traced back to the works of Hong-Kim-Pac\cite{Hong-1990} and Jackiw-Weinberg\cite{Jackiw-1990}. It is one of the classical partial differential equations commonly encountered in geometry and physics, and has attracted considerable attention from mathematicians. Many researchers have contributed to this field; see, for instance, \cite{2,3,9,31,41,43,46}, and the references therein.

Recently, there has been growing interest among scholars in the analysis of Chern-Simons Higgs models on graphs. In particular, Huang-Lin-Yau\cite{Huang-2020} and Hou-Sun\cite{Hou-2022} investigated the existence of solutions to the equation
\begin{equation*}\label{1.02}
\Delta u=\lambda e^u(e^u-1)+4\pi\sum_{i=1}^N \delta_{p_i},
\end{equation*}
where $N$ is a positive integer, $p_1,\cdots,p_N$ are arbitrarily chosen vertices on the graph and $\delta_{p_i}$ is the Dirac delta
mass at $p_i$. Their approach was based on the method of upper and lower solutions. Using the variational method, Chao-Hou\cite{Chao-2023} investigated the existence of solutions to the Chern-Simons Higgs models of the form
\begin{equation*}\label{1.002}
\Delta u=\lambda e^u(e^u-1)^5+4\pi\sum_{i=1}^N\delta_{p_i}.
\end{equation*}
For more related works, we refer the reader to \cite{Huang-2021,Li-CVPDE-2024} and the references therein.

It is well known that the topological degree theory is a powerful tool for studying partial differential equations in the Euclidean space and on Riemann surfaces; it can also be applied to deal with partial differential equations on finite graphs. This approach was first used by Sun-Wang\cite{Sun-adv-2022} to solve the Kazdan-Warner equation on finite graphs
\begin{equation*}\label{1.03}
\Delta u=c-h e^u,
\end{equation*}
where $h$ is a real function on $V$ and $c$ is a real constant. Recently, the topological degree theory was also employed by Liu\cite{Liu-2022} to study the mean field equation
\begin{equation*}\label{1.003}
\Delta u=\rho \left(\frac{h e^u}{\int_V he^ud\mu}-\frac{1}{|V|}\right)
\end{equation*}
with $\rho\in\mathbb{R}\setminus\{0\}$, $h:V\to\mathbb{R}^+$ being a function satisfying $\min_{x\in V}h(x)>0$, and $|V|$ denoting the cardinality of the vertex set $V$. Here
\[\int_{V} g d \mu=\sum_{x \in V} g(x)\]
denotes the integral of $g: V\to \mathbb{R}$. Moreover, Li-Sun-Yang\cite{Li-CVPDE-2024} also used it to investigate the Chern-Simons Higgs models on finite graphs of the form
\begin{equation*}\label{1.001}
\Delta u=\lambda e^u(e^u-1)+f.
\end{equation*}

Our aim is to apply this powerful tool to study the Chern-Simons Higgs models (\ref{1.1}) for the $p$-Laplacian. Since $\Delta_p$ is nonlinear, new challenges arise in calculating the topological degree of the corresponding map.

\medskip
It is well known that applying topological degree theory involves two crucial steps. The first step is to establish a priori bound for solutions, as stated in the following theorem.

\begin{theorem}\label{t1.1}
Let $G=(V, E)$ be a connected finite graph. Suppose $\sigma \in[0,1]$, and $\lambda$ and $f$ satisfy $\lambda\int_{V}fd\mu<0$. If $u$ is a solution of the equation
\begin{equation}\label{1.3}
\Delta_p u=\lambda e^{u}\left(e^{u}-\sigma\right)+f, \quad\hbox{in}~~ V,
\end{equation}
then there exists a constant $C$, depending only on $\lambda$, $f$ and $|V|$, such that $|u(x)| \leqslant C$ for all $x \in V$.
\end{theorem}

In Theorem \ref{t1.1} and the following two theorems, $\lambda$ is a real number, and $f$ is a fixed function on $V$. When $\sigma=1$, equation \eqref{1.3} is exactly \eqref{1.1}. Furthermore, we define the space $L^{\infty}(V)$ and its associated norm as
\[L^{\infty}(V) = \{ u : V \to \mathbb{R}: \|u\|_{L_\infty(V)}<+\infty \}, \]
and
 \[\|u\|_{L^{\infty}(V)} = \sup_{x \in V} |u(x)|,\]
respectively. Let $X=L^{\infty}(V)$, in which the zero element is denoted by $0$, and define a map $F: X \rightarrow X$ as
\begin{equation}\label{1.4}
F(u)=-\Delta_p u+\lambda e^{u}\left(e^{u}-1\right)+f.
\end{equation}
The second step is to calculate the topological degree of $F$ by utilizing its homotopy invariance property.

\begin{theorem}\label{t1.2}
Let $G=(V, E)$ be a connected finite graph, and let $F: X \rightarrow X$ be the map defined by \eqref{1.4}. Assume that $\lambda \int_{V} f d \mu < 0$. Then there exists a $R_{0}>0$ large enough such that for all $R \geqslant R_{0}$,
\[\operatorname{deg}\left(F, B_{R}, 0\right)=\operatorname{sgn}{\lambda}=
\left\{\begin{array}{lll}\vspace{0.1cm}\displaystyle
1, & \text {if}\quad\lambda>0, \\ \displaystyle
-1, & \text {if}\quad\lambda<0,
\end{array}\right.\]
where $B_{R}=\left\{u \in X:\|u\|_{L^{\infty}(V)}<R\right\}$ is the ball in $X$ centered at $0\in X$ with radius $R$, and $\deg(F, B_R, 0)$ is the topological degree of $F$ in $B_R$ at $0$.
\end{theorem}

As an application of the above theorem, we establish the following existence result for the Chern-Simons Higgs models involving $p$-Laplacian.

\begin{theorem}\label{t1.3}
Let $G=(V, E)$ be a connected finite graph, and $\lambda \int_{V} f d \mu<0$. Then equation \eqref{1.1} has a solution.
\end{theorem}

\begin{remark}\label{r1.1}
As is well known, the fundamental idea of the topological degree theory is to reduce a complex equation (or system) to a simple equation (or system) through the homotopy invariance of the topological degree. In general, the existence of solutions to the simple equation (or system) is either clear or can be proven. Then one can calculate that the topological degree of some operator associated to the simple equation (or system) is not zero. In Section 4, after reducing equation \eqref{1.1} to equation \eqref{3.4}, it becomes essential to investigate the existence of solutions to equation \eqref{3.4}, which constitutes the primary focus of Section 3.
\end{remark}

\begin{remark}\label{r1.2}
Many researchers focus on graphs equipped with weights and measures. all edges $xy \in E$ carry weights $w_{xy} > 0$ satisfying $w_{xy} = w_{yx}$. The measure on $V$ is defined as $\mu: V\rightarrow (0,+\infty)$. In this paper, we assume that $w\equiv1$ and $\mu\equiv 1$, as the specific choices of weights and measures are not essential for our analysis.
\end{remark}

In the following we will use $C$ to represent the corresponding constants, which may not necessarily be equal in the context.

\medskip
The rest of the paper is organized as follows. In Section 2, we provide a priori estimate of solutions for equation \eqref{1.1} (Theorem \ref{t1.1}). In Section 3, we present some preliminary lemmas. Finally, in Section 4, we prove Theorem \ref{t1.2} and Theorem \ref{t1.3}.

\section{A priori estimate}
\setcounter{equation}{0}

First, we introduce some basic definitions and preliminary lemmas. For any $x\in V$, the gradient of $u$ at $x$ is defined as
\[\nabla u(x)=\big(\left(u\left(y_{1}\right)-u(x)\right), \cdots, \left(u\left(y_{l_x}\right)-u(x)\right)\big),\]
where $y_{1}, \cdots, y_{l_{x}}$ are all distinct points adjacent to $x$, and $l_x$ denotes the number of edges incident to $x$. Clearly, such an $l_{x}$ is unique and $\nabla u(x) \in \mathbb{R}^{l_{x}}$. The integral of $|\nabla u|^p$ on $V$ is given by
\[\int_{V} |\nabla u|^p d \mu=\frac{1}{2}\sum_{x \in V}\sum_{y\sim x} |u(y)-u(x)|^p.\]

Furthermore, we define a Sobolev space and its associated norm as
\[W^{1,p}(V)=\left\{u:V\to\mathbb{R}:\int_V (|u|^p+|\nabla u|^p)d\mu<+\infty\right\},\]
and
\[\|u\|_{W^{1,p}(V)}=\left(\int_V(|u|^p+|\nabla u|^p)d\mu\right)^{1/p},\]
respectively. Since $G=(V,E)$ is a finite graph, then $W^{1,p}(V)$ is exactly the set of all real functions on $V$, a finite dimensional linear space. Consequently, one has the following two lemmas, which were proved in \cite{Ge-2020}.

\begin{lemma}\label{l3.5} Let $G=(V, E)$ be a finite graph. The Sobolev space $W^{1,p}(V)$ is precompact. Namely, if $\{u_n\}$ is bounded in $W^{1,p}(V)$, then there exists some $u\in W^{1,p}(V)$ such that, up to a subsequence, $u_n\to u$ in $W^{1,p}(V)$.
\end{lemma}

\begin{lemma}\label{l2.0}
Let $G=(V, E)$ be a finite graph. We use
 $$\bar{g}=\frac{1}{|V|}\int_V gd\mu$$
  to denote the integral mean of the function $g: V\rightarrow\mathbb{R}$. For all functions $u\in W^{1,p}(V)$ with $\bar{u}=0$, there exists a positive constant $C$ depending on $p$ and $|V|$ such that
\[\int_V|u|^pd\mu\leqslant C\int_V|\nabla u|^pd\mu.\]
\end{lemma}

Next, we proceed to the proof of Theorem \ref{t1.1}, following the main idea inspired by the proof of Theorem 1 in \cite{Li-CVPDE-2024}.

\begin{proof}[Proof of Theorem \ref{t1.1}]
Suppose that $u$ is a solution of equation \eqref{1.3}. By integrating on both sides of equation \eqref{1.3} on $V$, we have
\begin{equation}\label{2.1}
0=\int_{V} \Delta_p u d \mu=\lambda \int_{V} e^{u}\left(e^{u}-\sigma\right) d \mu+\int_{V} f d \mu.
\end{equation}
Next, we divide the proof into two steps.

\smallskip
\textbf{Step 1.} We prove that $u$ has a uniform upper bound.
Without loss of generality, we may assume $\max\limits _{V} u>0$. For otherwise, $u$ has already an upper bound 0. Note that
\[\left|\int_{u<0} e^{u}\left(e^{u}-\sigma\right) d \mu\right| \leqslant |V|.\]
Furthermore, by \eqref{2.1}, we can deduce that
\[\int_{u \geqslant 0} e^{u}\left(e^{u}-\sigma\right) d \mu \leqslant |V|+\frac{1}{|\lambda|}\left|\int_{V} f d \mu\right| \leqslant |V|\left(1+\frac{1}{|\lambda|}|\bar{f}|\right)=:a.\]
Here and hereafter, $\bar{f}=\frac{1}{|V|}\int_V f~d\mu$ denotes the integral mean of the function $f$. Together with
\[\int_{u \geqslant 0} e^{u}\left(e^{u}-\sigma\right) d \mu=\sum_{x \in V,u(x)\geqslant0} e^{u(x)}\left(e^{u(x)}-\sigma\right) \geqslant e^{\max\limits _{V} u}\left(e^{\max\limits _{V} u}-1\right),\]
it follows that
\begin{equation}\label{2.2}
\max _{V} u \leqslant \ln \frac{1+\sqrt{1+4 a}}{2}.
\end{equation}

\medskip
\textbf{Step 2.} We prove that $u$ has a uniform lower bound. For any $x \in V$, by \eqref{1.3} and \eqref{2.2}, keeping in mind $\sigma\in[0,1]$,  we have
\begin{equation*}
\begin{array}{ll}\vspace{0.5cm}\displaystyle|\Delta_p u(x)|&\leqslant|\lambda|\left|e^{u(x)}\left(e^{u(x)}-\sigma\right)\right|+|f(x)|
\\ \vspace{0.3cm}\displaystyle
&\leqslant|\lambda|\left(e^{2 u(x)}+e^{u(x)}\right)+|f(x)|\\ \vspace{0.3cm}\displaystyle
&\leqslant|\lambda|\big(\frac{\left(1+\sqrt{1+4 a}\right)^{2}}{4}+\frac{1+\sqrt{1+4 a}}{2}\big)+\|f\|_{L^{\infty}(V)}
\\ \displaystyle
&=: b .
\end{array}
\end{equation*}
Hence,
\begin{equation}\label{2.3}
\|\Delta_p u\|_{L^{\infty}(V)} \leqslant b.
\end{equation}

\medskip
We may assume $V=\left\{x_{1}, \cdots, x_{l}\right\}$, $u\left(x_{1}\right)=\max\limits _{V} u$, $u\left(x_{l}\right)=\min \limits_{V} u$, and without loss of generality $x_{1} x_{2}, x_{2} x_{3}, \cdots, x_{l-1} x_{l}$ is the shortest path connecting $x_{1}$ and $x_{l}$. Then, by Jensen's inequality, we can deduce that
\begin{align}
    0 \leqslant u\left(x_{1}\right)-u\left(x_{l}\right)&\leqslant \sum_{j=1}^{l-1}\left|u\left(x_{j}\right)-u\left(x_{j+1}\right)\right|\notag \\
    &\leqslant \big((l-1)^{p-1}\big)^{1/p}\left(\sum_{j=1}^{l-1} \left|u\left(x_{j}\right)-u\left(x_{j+1}\right)\right|^{p}\right)^{1/p}\notag \\
     &\leqslant\big(2(l-1)^{p-1}\big)^{1/p}\left(\int_{V}|\nabla u|^{p} d \mu\right)^{1/p}.\label{2.4} 
\end{align}
Since
\begin{eqnarray*}&&\int_{V}|\nabla u|^{p} d \mu+\int_V u \Delta_p u d \mu\\
 &=&\frac{1}{2}\sum_{x \in V} \sum_{y\sim x}|u(y)-u(x)|^{p-2}(u(y)-u(x))^{2}\\
&&+\sum_{x \in V} \sum_{y\sim x} u(x)|u(y)-u(x)|^{p-2}(u(y)-u(x))\\
&=&\frac{1}{2} \sum_{x \in V} \sum_{y\sim x}|u(y)-u(x)|^{p-2}(u(y)-u(x))(u(y)-u(x)+2 u(x))\\
&=&\frac{1}{2} \sum_{x \in V} \sum_{y\sim x}|u(y)-u(x)|^{p-2}(u(y)-u(x))(u(y)+u(x))\\
&=&0,
\end{eqnarray*}
by Lemma \ref{l2.0}, it holds
\[\begin{aligned}
\int_{V}|\nabla u|^{p} d \mu & =-\int_{V}(u-\bar{u}) \Delta_p u d \mu \\
& \leqslant\left(\int_{V}|u-\bar{u}|^{p} d \mu\right)^{1 / p}\left(\int_{V}|\Delta_p u|^{q} d \mu\right)^{1 / q} \\
& \leqslant\left(C \int_{V}|\nabla u|^{p} d \mu\right)^{1 / p}\left(\int_{V}|\Delta_p u|^{q} d \mu\right)^{1/q},
\end{aligned}\]
where $q$ is a positive integer satisfying $\dfrac{1}{p}+\dfrac{1}{q}=1$. Thus, we have
\begin{equation}\label{2.5}
\int_{V}|\nabla u|^{p} d \mu \leqslant C \int_{V}|\Delta_p u|^{q} d \mu \leqslant C\|\Delta_p u\|_{L^{\infty}(V)}^{q}|V|.
\end{equation}
Combining \eqref{2.4} and \eqref{2.5}, we can conclude
\begin{equation}\label{2.6}
\max _{V} u-\min _{V} u \leqslant \left(2(l-1)^{p-1}|V|C\right)^{1/p}\|\Delta_p u\|_{L^{\infty}(V)}^{q/p}.
\end{equation}
Then from \eqref{2.3}, we obtain
\begin{equation}\label{2.7}
\max _{V} u-\min _{V} u \leqslant c_{0}:=b^{q/p} \left(2(l-1)^{p-1}|V|C\right)^{1/p}.
\end{equation}
Returning to \eqref{2.1}, it follows that
\begin{equation}\label{2.8}
\int_{V} e^{u}\left(e^{u}-\sigma\right) d \mu=-\frac{1}{\lambda} \int_{V} f d \mu=-\frac{1}{\lambda}\bar{f}|V| =:\zeta.
\end{equation}
Since $\lambda\int_{V} f d \mu<0$, then $\zeta>0$. We claim that
\begin{equation}\label{2.9}
\max _{V} u>-A:=\ln \min \left\{1, \frac{\zeta}{4|V|}\right\}.
\end{equation}
For otherwise, $\max _{V} u \leqslant-A$, which, together with \eqref{2.8}, implies
\begin{equation*}
\begin{array}{ll}\vspace{0.4cm}\displaystyle \zeta=-\frac{1}{\lambda}\bar{f}|V|=\left|\int_{V} e^{u}\left(e^{u}-\sigma\right) d \mu\right|\leqslant \int_{V}\left(e^{2 u}+e^{u}\right) d \mu\\
\vspace{0.2cm}\displaystyle\qquad\qquad\leqslant\left(e^{2 \max\limits _{V} u}+e^{\max\limits _{V} u}\right)|V|\leqslant 2 e^{-A}|V| <\frac{\zeta}{2},
\end{array}
\end{equation*}
which is impossible. This confirms our claim. Inserting \eqref{2.9} into \eqref{2.7}, we have
\[-A-c_{0} \leqslant \min _{V} u \leqslant \max _{V} u \leqslant \ln \frac{1+\sqrt{1+4 a}}{2}.\]
\end{proof}

\section{Preliminary lemmas}
\setcounter{equation}{0}

In this section, we present some lemmas in preparation for the proof of Theorem 1.2 in the next section. Our main focus is on the existence of solutions to the equation
\begin{equation}\label{33}
\Delta_p u=\lambda e^{2u}+\varepsilon f,
\end{equation}
where $\varepsilon>0$ is a positive constant.

\subsection{The case $\lambda>0$}

\medskip
First, we introduce the definitions of the lower and upper solutions to equation \eqref{33}.

\medskip
We call a function $u_-$ a lower solution to equation \eqref{33} if, for all $x\in V$, it satisfies
\[\Delta_p u_-\geqslant\lambda e^{2u_-}+\varepsilon f.\]

Similarly, a function $u_+$ is called an upper solution to equation \eqref{33} if, for all $x\in V$, it satisfies
\[\Delta_p u_+\leqslant\lambda e^{2u_+}+\varepsilon f.\]

In the following, Lemma \ref{l4.1} concerns the method of upper and lower solutions, which serves as a preparation for the proof of Lemma \ref{l3.2}. Lemma \ref{l3.2} establishes the existence of a solution to equation \eqref{33} when $\lambda>0$.

\begin{lemma}\label{l4.1} Let $\lambda>0$ and $\int_V fd\mu<0$. If there exist lower and upper solutions $u_-$ and $u_+$ to the equation \eqref{33} satisfying $u_-\leqslant u_+$, then equation \eqref{33} has a solution $u$ satisfying $u_- \leqslant u \leqslant u_+$.
\end{lemma}

\begin{proof}
Set $k=2\lambda e^{2u_+}> 0$. We define
\[L\varphi\equiv \Delta_p\varphi-k\varphi.\]
Set $u_0= u_+$, and for each integer $n\geqslant0$,
\[Lu_{n+1}=\lambda e^{2u_{n}}+\varepsilon f-ku_n.\]
We claim that
\[u_-\leqslant u_{n+1} \leqslant u_n \leqslant \cdots \leqslant u_1 \leqslant u_0 = u_+.\]
To prove the above claim, we just need to prove
\[Lu_-\geqslant Lu_{n+1}\geqslant Lu_{n}\geqslant\cdots \geqslant Lu_1 \geqslant Lu_0 = Lu_+,\]
according to the order-preserving property of $L$ in Theorem 2.1 in \cite{Ge-2020}, which is, if $Lu\geqslant Lv$, then $u\leqslant v$. We prove this by induction. It is straightforward to deduce that
\[L u_1 - L u_0=\lambda e^{2u_0}+\varepsilon f-ku_0-Lu_+ \geqslant 0.\]
By the mean value theorem, we have
\begin{equation*}
\begin{array}{ll}\vspace{0.3cm}\displaystyle L u_--Lu_{1} =( \Delta_p u_- - k u_-)-(\lambda e^{2u_+}+\varepsilon f- ku_+)\\ \vspace{0.3cm}\displaystyle
\qquad\qquad\quad\geqslant k ( u _ { + } - u _ { - }) -\lambda ( e^{2 u_+ }-e ^ {2 u - })\\ \vspace{0.3cm}\displaystyle
\qquad\qquad\quad= k( u_+ - u_-) -2\lambda e^{2\xi}(u_+ - u_-)\\ \vspace{0.3cm}\displaystyle
\qquad\qquad\quad= 2\lambda(u_+ -u_-)(e^{2u_+}-e^{2\xi})\\ \displaystyle
\qquad\qquad\quad\geqslant 0,
\end{array}
\end{equation*}
where $u_-\leqslant \xi \leqslant u_+$. Thus, we have
\[L u_-\geqslant L u_1\geqslant L u_0= L u_+.\]
Suppose
\[L u_-\geqslant L u_n\geqslant L u_{n-1}\geqslant L u_+,\]
for some $n\geqslant1$. Then
\begin{equation*}
\begin{array}{ll}\vspace{0.3cm}\displaystyle Lu_{n+1}-Lu_n=\lambda(e^{2u_n}-e^{2u_{n-1}})-k(u_n-u_{n-1}) \\ \vspace{0.3cm}\displaystyle
\qquad\qquad\quad\quad=2\lambda e^{2\xi}(u_{n}-u_{n-1})-k(u_n-u_{n-1})\\ \vspace{0.3cm}\displaystyle
\qquad\qquad\quad\quad=2\lambda (e^{2\xi}-e^{2u_+})(u_{n}-u_{n-1})\\ \displaystyle
\qquad\qquad\quad\quad\geqslant0,
\end{array}
\end{equation*}
where $u_n\leqslant\xi\leqslant u_{n-1}\leqslant u_+$. Similarly,
\begin{equation*}
\begin{array}{ll}\vspace{0.3cm}\displaystyle Lu_--Lu_{n+1}=(\Delta_pu_- -ku_-)-(\lambda e^{2u_n}+\varepsilon f-ku_n)\\ \vspace{0.3cm}\displaystyle
\qquad\qquad\quad\geqslant k(u_n-u_-)-\lambda(e^{2u_n}-e^{2u_-})\\ \vspace{0.3cm}\displaystyle
\qquad\qquad\quad= k(u_n-u_-)-2\lambda e^{2\xi}(u_n-u_-)\\ \vspace{0.3cm}\displaystyle
\qquad\qquad\quad=2\lambda (e^{2u_+}-e ^{2\xi})(u_n-u_-)\\ \displaystyle
\qquad\qquad\quad\geqslant 0 ,
\end{array}
\end{equation*}
where $u_-\leqslant \xi \leqslant u_n\leqslant u_+$. Hence we obtain
\[L u_-\geqslant Lu_{n+1}\geqslant Lu_n\geqslant Lu_+.\]
By induction, the claim above has been proven. Since $V$ is finite, by Lemma \ref{l3.5}, there exists a $u$, up to a subsequence, $u_n\rightarrow u$ as $n \rightarrow \infty$ uniformly on $V$. Note that
\[Lu_{n+1}=\Delta_pu_{n+1}-k u_{n+1}=\lambda e^{2u_n}+\varepsilon f-ku_n,\]
let $n\rightarrow\infty$, we have
\[\Delta_p u=\lambda e^{2u}+\varepsilon f.\]
Hence, $u$ is a solution to the equation \eqref{33}.
\end{proof}

\begin{lemma}\label{l3.2} Let $G=(V, E)$ be a finite graph. Assume that $\lambda$ and $f$ satisfy $\lambda>0$ and $\int_V fd\mu<0$, respectively. Then, for any given $\varepsilon>0$, equation \eqref{33} has a solution.
\end{lemma}

\begin{proof}
~Let $v_\varepsilon$ be a solution of the equation (The solvability of the equation can be found in \cite{Ge-2020})
\[\Delta_pv=\varepsilon f-\varepsilon\bar{f}.\]

Consider $u_-=v_\varepsilon-A$ with $A>0$, we have
\begin{equation*}
\begin{array}{ll}\vspace{0.3cm}\displaystyle \Delta_pu_--\lambda e^{2u_-}-\varepsilon f=\Delta_pv_\varepsilon-\lambda e^{2(v_\varepsilon-A)}-\varepsilon f
\\ \vspace{0.3cm}\displaystyle
\qquad\qquad\qquad\qquad~=-\varepsilon\bar{f}-\lambda e^{2v_\varepsilon}e^{-2A}\\ \displaystyle
\qquad\qquad\qquad\qquad~\to-\varepsilon\bar{f}, \quad \hbox{as}~A\to+\infty,
\end{array}
\end{equation*}
uniformly for $x\in V$. Noting that $\bar{f}<0$, we can find a sufficiently large $A$ such that
\[\Delta_pu_--\lambda e^{2u_-}-\varepsilon f\geqslant0.\]
Therefore, equation \eqref{33} has a lower solution $u_-$.

\medskip
Consider $u_+ =v_\varepsilon+b$ with $b>0$. Noting that $\lambda>0$, we have
\begin{equation*}
\begin{array}{ll}\vspace{0.3cm}\displaystyle \Delta_pu_+-\lambda e^{2u_+}-\varepsilon f
=\Delta_pv_\varepsilon-\lambda e^{2(v_\varepsilon+b)}-\varepsilon f \\ \vspace{0.3cm}\displaystyle
\qquad\qquad\qquad\qquad~=-\varepsilon\bar{f}-\lambda e^{2v_\varepsilon}e^{2b}\\ \displaystyle
\qquad\qquad\qquad\qquad~\to-\infty, \quad \hbox{as}~b\to+\infty,
\end{array}
\end{equation*}
uniformly for $x\in V$. We can therefore find a sufficiently large $b$ such that
\[\Delta_pu_+-\lambda e^{2u_+}-\varepsilon f\leqslant0.\]
Hence, equation \eqref{33} has an upper solution $u_+$. By Lemma \ref{l4.1}, we can conclude that equation \eqref{33} has a solution $u$ satisfying $u_-\leqslant u\leqslant u_+$. The proof is completed.
\end{proof}

\subsection{The case $\lambda<0$}

\medskip
In this subsection, we consider the existence of solutions to equation \eqref{33} when $\lambda<0$. The main result is stated in Lemma \ref{l3.1}, while Lemmas \ref{l3.4} and \ref{l3.3} serve as preparatory results for its proof.

\medskip
For all $u\in W^{1,p}(V)$, let
\[\bar{u}_f=\frac{\int_Vfud\mu}{\int_Vfd\mu}=\frac{1}{\bar{f}|V|}\int_Vfud\mu.\]
We begin with the following result, whose proof is similar to that of Lemma \ref{l2.0} and is therefore omitted.

\begin{lemma}\label{l3.4} Let $G=(V, E)$ be a finite graph. For all functions $u\in W^{1,p}(V)$ with $\bar{u}_f=0$, there exists a positive constant $C$ depending on $p$, $f$ and $|V|$ such that
\[\int_V|u|^pd\mu\leqslant C\int_V|\nabla u|^pd\mu.\]
\end{lemma}

\begin{lemma}\label{l3.3} Let $G=(V, E)$ be a finite graph. For any $\beta\geqslant0$ and $\alpha\geqslant0$, there exist two constants $C^*$ (depending on $p$, $f$ and $|V|$) and $C_*$ (depending on $\alpha$, $\beta$, $p$, $f$ and $|V|$) such that for all functions $u$ satisfying $\int_V|\nabla u|^pd\mu\leqslant1$ and $\bar{u}_f=0$, it holds that
\[\|u\|_\infty\leqslant C^*,\]
and
\[\int_V e^{\beta |u|^\alpha}d\mu\leqslant C_*.\]
\end{lemma}

\begin{proof}
For any function $u$ satisfying $\int_V|\nabla u|^pd\mu\leqslant 1$ and $\bar{u}_f=0$, by Lemma \ref{l3.4}, we have
\begin{equation*}\label{3.01}
\begin{array}{ll}\vspace{0.3cm}\displaystyle \|u\|_\infty\leqslant\sum_{x\in V}|u(x)|=\int_V|u|d\mu\leqslant\left(\int_V|u|^pd\mu\right)^{1/p}\left(\int_V1^qd\mu\right)^{1/q}\\ \displaystyle
\qquad\quad\leqslant\left( C\int_V|\nabla u|^pd\mu\right)^{1/p}|V|^{1/q}\leqslant C^{1/p}|V|^{1/q}:=C^*,
\end{array}
\end{equation*}
where $q$ is a positive integer satisfying $\dfrac{1}{p}+\dfrac{1}{q}=1$. Since $\beta\geqslant0$ and $\alpha\geqslant0$, we have
\[\int_V e^{\beta |u|^\alpha}d\mu\leqslant\int_V e^{\beta \|u\|_\infty^\alpha}d\mu\leqslant e^{\beta {C^*}^{\alpha}}|V|:=C_*.\]
\end{proof}

\begin{lemma}\label{l3.1} Let $G=(V, E)$ be a finite graph. Assume that $\lambda$ and $f$ satisfy $\lambda<0$ and $\bar{f}>0$, respectively. Then, for any given $\varepsilon>0$, equation \eqref{33} has a solution.
\end{lemma}

\begin{proof}
Define a set
\[\mathcal{B}= \left\{u\in W^{1,p}(V): \lambda \int_V e^{2u}d\mu+\varepsilon\bar{f}|V|=0\right\}.\]
We claim that $\mathcal{B}\neq \emptyset$.

\medskip
Next, we prove the above claim. Define a function
\[u_l(x)=l, \quad \forall ~x\in V.\]
Since $\lambda<0$, then we have
\[ -\lambda\int_V e^{2u_l}d\mu=-\lambda e^{2l}|V|\to +\infty,\quad\hbox{as}~~l\to+\infty.\]
We also define $\tilde{u}_l\equiv -l$, which leads to
\[-\lambda\int_V e^{2\tilde{u}_l}d\mu =-\lambda e^{-2l}|V|\to0, \quad \hbox{as}~l \to+ \infty. \]
Hence, there exists a sufficiently large $l$ such that
\[-\lambda\int_V e^{2u_l}d\mu >\varepsilon\bar{f}|V|\quad \hbox{and}\quad -\lambda\int_V e^{2\tilde{u}_l}d\mu <\varepsilon\bar{f}|V|.\]
We define a function $\phi:\mathbb{R}\rightarrow\mathbb{R}$ by
\[\phi(t)= -\lambda\int_V e^{2\left(tu_l+(1-t)\tilde{u}_l\right)}d\mu .\]
Then,
\[\phi(0) < \varepsilon\bar{f} |V| < \phi(1) ,\]
and thus there exists a $t_0\in(0,1)$ such that $\phi(t_0)=\varepsilon\bar{f}|V|$. Therefore, we have shown that $\mathcal{B} \neq \emptyset$.

\medskip
We solve \eqref{33} by minimizing the functional
\[I(u)=\frac{1}{p} \int_V |\nabla u|^p d\mu + \varepsilon\int_V fu d\mu\]
on $\mathcal{B}$. For any $u\in\mathcal{B}$, we have
\[\int_V e^{2u-2\bar{u}_f}d\mu=e^{-2\bar{u}_f}\int_V e^{2u}d\mu=e^{-2\bar{u}_f}\left(-\frac{\varepsilon}{\lambda}\bar{f}|V|\right),\]
and thus
\begin{equation*}\label{3.01}
\begin{array}{ll}\vspace{0.3cm}\displaystyle I(u)=\frac{1}{p} \int_V |\nabla u|^pd\mu+\varepsilon\int_V fu d\mu
-\frac{1}{2}\varepsilon\int_V f\ln\left(\int_V e^{2u-2\bar{u}_f}d\mu\right)d\mu\\ \vspace{0.3cm}\displaystyle
\qquad\quad+\frac{1}{2}\varepsilon\int_V f\ln\left(\int_V e^{2u-2\bar{u}_f}d\mu\right)d\mu\\ \vspace{0.3cm}\displaystyle
\qquad=\frac{1}{p} \int_V |\nabla u|^pd\mu-\frac{1}{2}\varepsilon\int_V f\ln\left(\int_V e^{2u-2\bar{u}_f}d\mu\right)d\mu\\ \vspace{0.3cm}\displaystyle
\qquad\quad+\varepsilon\int_V fu d\mu+\frac{1}{2}\varepsilon\int_V f\left[-2\bar{u}_f+\ln\left(-\frac{\varepsilon}{\lambda}\bar{f}|V|\right)\right]d\mu.
\end{array}
\end{equation*}
Note that
\[\int_Vfud\mu=\bar{u}_f\int_Vfd\mu,\]
then, we have
\begin{equation*}
I(u)=\frac{1}{p} \int_V |\nabla u|^pd\mu-\frac{1}{2}\varepsilon\int_V f\ln\left(\int_V e^{2u-2\bar{u}_f}d\mu\right)d\mu+\frac{1}{2}\varepsilon\bar{f}|V|\ln\left(-\frac{\varepsilon}{\lambda}\bar{f}|V|\right).
\end{equation*}
Set
\[w=\frac{u-\bar{u}_f}{\|\nabla (u-\bar{u}_f)\|_p}=\frac{u-\bar{u}_f}{\|\nabla u\|_p},\]
then
\begin{equation*}
\begin{array}{ll}\vspace{0.3cm}\displaystyle\bar{w}_f=\frac{1}{\bar{f}|V|}\int_Vwfd\mu=\frac{1}{\bar{f}|V|}
\int_V\frac{(u-\bar{u}_f)f}{\|\nabla (u-\bar{u}_f)\|_p}d\mu=0,
\end{array}
\end{equation*}
and $\|\nabla w\|_p^p=1$. Furthermore, by Lemma \ref{l3.3}, for any $\beta\geqslant0$ and $\alpha\geqslant0$, one can find a constant $C_*$
depending on $\alpha$, $\beta$, $p$, $f$ and $|V|$ such that
\begin{equation*}\label{3.02}
\int_V e^{\beta |w|^\alpha}d\mu \leqslant C_*.
\end{equation*}
On the other hand, by an elementary inequality for $\delta>0$ and conjugate exponents $p, q$ ($\dfrac{1}{p} + \dfrac{1}{q} = 1$):
\[ab\leqslant\dfrac{\delta^p}{p} a^p+ \dfrac{\delta^{-q}}{q}b^q, \quad \forall a,b\geqslant0,\]
we have
\[u-\bar{u}_f=\|\nabla u\|_pw\leqslant\frac{\delta^p}{p} \|\nabla u\|_p^p+ \frac{\delta^{-q}}{q}|w|^q.\]
Then, it follows that
\begin{equation*}
\begin{array}{ll}\vspace{0.3cm}\displaystyle\int_V e ^{2u-2\bar{u}_f}d\mu\leqslant\int_V e ^{2\left(\frac{\delta^p}{p} \|\nabla u\|_p^p+ \frac{\delta^{-q}}{q}|w|^q\right)}d\mu=e^{\frac{2\delta^p}{p}\|\nabla u\|_p^p}\int_V e^{\frac{2\delta^{-q}}{q}|w|^q}d\mu\leqslant Ce^{\frac{2\delta^p}{p}\|\nabla u\|_p^p},
\end{array}
\end{equation*}
where $C$ is a positive constant depending on $p$, $\delta$, $f$ and $|V|$. The above inequality leads to
\begin{equation*}
\begin{array}{ll}\displaystyle I(u)\geqslant\frac{1}{p} \int_V|\nabla u|^p d\mu
-\varepsilon\bar{f}|V|\frac{\delta^p}{p}\|\nabla u\|_p^p-C_1,
\end{array}
\end{equation*}
where
\[C_1=\dfrac{1}{2}\varepsilon\bar{f}|V|\ln C-\dfrac{1}{2}\varepsilon\bar{f}|V|\ln\left(-\dfrac{\varepsilon}{\lambda}\bar{f}|V|\right).\]
Choosing $\delta=\left(2\varepsilon\bar{f}|V|\right)^{-1/p}$, we obtain for all $u\in\mathcal{B}$,
\begin{equation}\label{3.03}
I(u)\geqslant\frac{1}{2p}\int_V |\nabla u|^p d\mu-C_1.
\end{equation}
Therefore, $I$ has a lower bound on the set $\mathcal{B}$. This allows us to define
\[\hat{b}=\inf_{u\in\mathcal{B}} I(u) .\]

Take a sequence of functions $\{u_n\} \subset \mathcal{B}$ such that $I(u_n)\to \hat{b}$. We can suppose that
\[I(u_n)\leqslant \hat{b}+1,\]
for every $n$. It follows from \eqref{3.03} that
\[\|\nabla u_n\|_p^p\leqslant 2p\left(I(u_n)+C_1\right)\leqslant 2p\left(\hat{b}+1+C_1\right).\]
Since $u_n\in\mathcal{B}$, it follows that
\begin{align}\vspace{0.3cm}\displaystyle |{\overline{u_n}}_f|&=\dfrac{|\int_Vfu_nd\mu|}{\int_Vfd\mu}\notag\\ \vspace{0.3cm}\displaystyle
&=\dfrac{1}{\varepsilon\bar{f}|V|}\left|I(u_n)-\dfrac{1}{p}\int_V |\nabla u|^p d\mu\right|\notag\\ \displaystyle
&\leqslant \dfrac{1}{\varepsilon\bar{f}|V|}\bigg(3(\hat{b}+1)+2C_1\bigg).\label{bound-1}
\end{align}
Note that $\int_V(u_n-{\overline{u_n}}_f)fd\mu=0$, by Lemma \ref{l3.4}, we have
\begin{equation}\label{bound-2}
\|u_n-{\overline{u_n}}_f\|_p^p\leqslant C\int_V|\nabla(u_n-{\overline{u_n}}_f)|^pd\mu
=C\|\nabla u_n\|_p^p.
\end{equation}
Thus, by \eqref{bound-1} and \eqref{bound-2}, we have
\[\|u_n\|_p\leqslant\|u_n-{\overline{u_n}}_f\|_p+\|{\overline{u_n}}_f\|_p\leqslant C_2,\]
where $C_2$ is a positive constant depending on $\hat{b}$, $p$, $\delta$, $f$ and $|V|$. Therefore, $\{u_n\}$ is bounded in $W^{1,p}(V)$. Since $V$ is finite, by Lemma \ref{l3.5}, there exists some $u\in W^{1,p}(V)$ such that, up to a subsequence, $u_n\to u$ in $W^{1,p}(V)$. It is easy to see that $u\in\mathcal{B}$ and $I(u)=\hat{b}$.

\medskip
Finally, we derive the Euler-Lagrange equation of $I(u)$. For any $h\in W^{1,p}(V)$, it holds
\begin{equation*}
\begin{array}{ll}\vspace{0.3cm}\displaystyle 0=\frac{d}{dt}\bigg|_{t=0}\bigg\{I(u+th)+\tau\left(\int_V\lambda e^{2(u+th)}d\mu+\varepsilon\bar{f}|V|\right)\bigg\}\\ \vspace{0.3cm} \displaystyle
\quad=\int_V \left(-\Delta_pu+2\lambda\tau e^{2u}+\varepsilon f\right)h d\mu,
\end{array}
\end{equation*}
where $\tau$ is a constant. Thus, we have
\begin{equation}\label{zz01}
\Delta_p u=2\lambda\tau e^{2u}+\varepsilon f.
\end{equation}
Integrating on both sides of equation \eqref{zz01}, we obtain $\tau=1/2$. Then $u$ is a solution of equation \eqref{33}.
\end{proof}

\section{Topological degree and the existence result}
\setcounter{equation}{0}

In this section, we prove Theorem \ref{t1.2} and Theorem \ref{t1.3}. We first calculate the topological degree of certain maps related to the Chern-Simons Higgs models for $p$-Laplacian.

\begin{proof}[Proof of Theorem \ref{t1.2}]
Let $V=\left\{x_{1}, \cdots, x_{l}\right\}$ be a finite set, and define $X=L^{\infty}(V)$. We may identify $X$ with the Euclidean space $\mathbb{R}^{l}$. To apply the homotopy invariance of the topological degree without introducing ambiguity, we define a map $F: X \times[0,1] \rightarrow X$ by
\[F(u, \sigma)=-\Delta_p u+\lambda e^{u}\left(e^{u}-\sigma\right)+f, \quad(u, \sigma) \in X \times[0,1] .\]
Then we can conclude that $F$ is well-defined. In fact, if $p\geqslant2$, it is easy to see that $F$ is well-defined. If $1<p<2$, for any $y_{*}\sim x$ satisfying $u(y_*)=u(x)$, we can remove such $y_{*}$ from the summation, leading to
\[\Delta_{p} u(x)=\sum_{\{y:~y \sim x \}\setminus \{y_*|u(y_*)=u(x)\}}|u(y)-u(x)|^{p-2}(u(y)-u(x)),\]
which ensures that $F$ remains a well-defined map.

For the real number $\lambda$ and the fixed function $f$, since $\lambda \bar{f}<0$, then by applying Theorem \ref{t1.1}, we can conclude that there exists a constant $R_*>0$, depending only on $\lambda$, $f$ and $|V|$, such that for all $\sigma \in[0,1]$, all solutions of $F(u, \sigma)=0$ satisfy \[\|u\|_{L^{\infty}(V)}<R_*.\]
Denote by $B_{r} \subset X$ a ball centered at $0\in X$ with radius $r$, and its boundary by $\partial B_{r}=\left\{u \in X:\|u\|_{L^{\infty}(V)}=r\right\}$. Thus we can conclude that
\[0 \notin F\left(\partial B_{R}, \sigma\right), \quad \forall~\sigma \in[0,1],~ \forall~R \geqslant R_*.\]
By the homotopy invariance of the topological degree, we have
\begin{equation}\label{3.2}
\operatorname{deg}\left(F(\cdot, 1), B_{R}, 0\right)=\operatorname{deg}\left(F(\cdot, 0), B_{R}, 0\right), \quad \forall~R \geqslant R_*.
\end{equation}
Given any $\varepsilon>0$, we define another map $G_{\varepsilon}: X \times[0,1] \rightarrow X$ by
\[G_{\varepsilon}(u, t)=-\Delta_p u+\lambda e^{2 u}+(t+(1-t) \varepsilon) f, \quad(u, t) \in X \times[0,1] .\]
As with $F$, the map $G_\varepsilon$ is well-defined. Since $\lambda\int_Vfd\mu<0$, by applying Theorem \ref{t1.1} again, there exists a constant $R_{\varepsilon}^*>0$, depending only on $\varepsilon$, $\lambda$, $f$ and $|V|$, such that any solution $u$ of $G_{\varepsilon}(u, t)=0$ satisfy $\|u\|_{L^{\infty}(V)}<R_{\varepsilon}^*$ for all $t \in[0,1]$. This implies
\[0 \notin G_{\varepsilon}\left(\partial B_{R_\varepsilon}, t\right), \quad \forall~t \in[0,1],~\forall~R_\varepsilon\geqslant R_{\varepsilon}^*.\]
Hence, by the homotopy invariance of the topological degree, we have
\begin{equation}\label{3.3}
\operatorname{deg}\left(G_{\varepsilon}(\cdot, 1), B_{R_\varepsilon}, 0\right)=\operatorname{deg}\left(G_{\varepsilon}(\cdot, 0), B_{R_\varepsilon}, 0\right).
\end{equation}

To calculate $\operatorname{deg}\left(G_{\varepsilon}(\cdot, 0), B_{R_\varepsilon}, 0\right)$, we need to understand the solvability of the equation
\begin{equation}\label{3.4}
G_{\varepsilon}(u, 0)=-\Delta_p u+\lambda e^{2 u}+\varepsilon f=0.
\end{equation}
Now, we claim that: there exists an $\varepsilon_{0}>0$ such that for any $\varepsilon \in\left(0, \varepsilon_{0}\right)$, equation \eqref{3.4} has a unique solution $u_{\varepsilon}$, which satisfies $e^{2 u_{\varepsilon}} \leqslant \eta\varepsilon$, where $\eta$ is a constant depending only on $\lambda$, $f$ and $|V|$.

First, we prove the existence of the solution for equation \eqref{3.4}. Note that, under the assumptions that $\varepsilon>0$ and $\lambda\bar{f}<0$, by Lemma \ref{l3.2} and Lemma \ref{l3.1}, we can deduce that equation \eqref{3.4} has a solution $u_\varepsilon$. Moreover, integrating on both sides of \eqref{3.4}, we have
\begin{equation}\label{degg}
\int_{V} e^{2 u_{\varepsilon}} d \mu= -\frac{\varepsilon}{\lambda}\bar{f}|V|.
\end{equation}
It follows that
\begin{equation}\label{3.07}
e^{2 u_\varepsilon(x)} \leqslant \eta\varepsilon, \quad \forall~x \in V,
\end{equation}
where we take the constant $\eta\geqslant-\dfrac{1}{\lambda}\bar{f}|V|>0$.

\medskip
Second, we prove the uniqueness of the solution $u_\varepsilon$. Let $\varphi$ be an arbitrary solution of \eqref{3.4}, namely it satisfies
\begin{equation}\label{3.7}
\Delta_p \varphi=\lambda e^{2 \varphi}+\varepsilon f.
\end{equation}
Similar to \eqref{degg} and \eqref{3.07}, we have
\begin{equation}\label{3.8}
\int_{V} e^{2 \varphi} d \mu \leqslant \eta\varepsilon, \quad e^{2 \varphi(x)} \leqslant \eta\varepsilon\quad \text { for all } x \in V.
\end{equation}
Subtracting \eqref{3.7} from \eqref{3.4} and integrating on both sides, we have
\[0=\int_{V} (\Delta_p u_{\varepsilon}-\Delta_p\varphi) d \mu=\lambda \int_{V}\left(e^{2 u_{\varepsilon}}-e^{2 \varphi}\right) d \mu,\]
which implies
\[\min _{V}\left(u_{\varepsilon}-\varphi\right)<0<\max _{V}\left(u_{\varepsilon}-\varphi\right) .\]
Then, it follows that
\begin{equation}\label{3.9}
\left|u_{\varepsilon}-\varphi\right| \leqslant \max _{V}\left(u_{\varepsilon}-\varphi\right)-\min _{V}\left(u_{\varepsilon}-\varphi\right).
\end{equation}
By \eqref{3.4}, \eqref{3.07}, \eqref{3.7}, and \eqref{3.8}, we have
\begin{align}\label{3.10}
\vspace{0.3cm}\displaystyle\left|\Delta_p u_{\varepsilon}(x)-\Delta_p\varphi(x)\right| & =\left|\lambda\left(e^{2 u_{\varepsilon}(x)}-e^{2 \varphi(x)}\right)\right| \notag\\ \vspace{0.3cm}\displaystyle
& \leqslant 2|\lambda|\left(e^{2 u_{\varepsilon}(x)}+e^{2 \varphi(x)}\right)\left|u_{\varepsilon}(x)-\varphi(x)\right|\notag\\ \displaystyle
& \leqslant 4\eta|\lambda|\varepsilon\left|u_{\varepsilon}(x)-\varphi(x)\right|.
\end{align}
Combining \eqref{2.6}, \eqref{3.9} and \eqref{3.10}, we obtain
\begin{equation}\label{3.11}
\max _{V}\left(u_{\varepsilon}-\varphi\right)-\min _{V}\left(u_{\varepsilon}-\varphi\right) \leqslant C_3\varepsilon \left(\max _{V}\left(u_{\varepsilon}-\varphi\right)-\min _{V}\left(u_{\varepsilon}-\varphi\right)\right),
\end{equation}
where
\[C_3=4\eta|\lambda|\left(2(l-1)^{p-1}|V|C\right)^{1/p}.\]
Choose $\varepsilon_0=1/(2C_3)$.
If we take $0<\varepsilon<\varepsilon_{0}$, then \eqref{3.11} implies $\varphi \equiv u_{\varepsilon}$ on $V$. Thus \eqref{3.4} has a unique solution. Hence our claim holds.

We now proceed to prove the theorem. We calculate the topological degree of $G_\varepsilon$ at $u_\varepsilon$.
Note that $-\dfrac{\partial\Delta_pu}{\partial u}: X \rightarrow X$ is a symmetric operator. In fact, as before, we assume that $V=\{x_i|i=1,2,\cdot\cdot\cdot,l\}$, then we can regard $-\Delta_{p} u$ as a vector, i.e.,
\begin{equation*}
\begin{array}{ll}\vspace{0.3cm}\displaystyle-\Delta_{p} u=\big(-\Delta_{p} u(x_{1}), \cdots,-\Delta_{p} u(x_{\ell})\big)
\\ \displaystyle
\qquad\quad=\bigg(\sum_{y\sim x_{1}}|u(y)-u(x_{1})|^{p-2}(u(x_{1})-u(y)),\\ \displaystyle
\qquad\qquad\quad  \cdots,\sum_{y \sim x_{\ell}}|u(y)-u(x_{\ell})|^{p-2}(u(x_{\ell})-u(y)\bigg).
\end{array}
\end{equation*}
Denote
\begin{equation}\label{3.12}
 -\dfrac{\partial\Delta_pu}{\partial u}=\left[\begin{array}{cccc}
a_{11} & a_{12} & \cdots & a_{1 l} \\
a_{21} & a_{22} & \cdots & a_{2 l} \\
\vdots & & &  \\
a_{l 1} & a_{l 2} & \cdots & a_{ll}
\end{array}\right].
\end{equation}
Then, if $p\geqslant2$, we can calculate that
\[a_{ii}=(p-1)\sum_{y \sim x_{i}}\left|u(y)-u\left(x_{i}\right)\right|^{p-2},i=1,2,\cdots,l,\]
and
\[a_{ij}=a_{ji}=-(p-1)\left|u\left(x_{j}\right)-u\left(x_{i}\right)\right|^{p-2} \quad\text { or }\quad a_{ij}=a_{ji}=0, \]
for $i,j=1,2,\cdots,l$ and $i\neq j$. Moreover, the following equalities hold:
\[a_{i1}+a_{i2}+\cdots+a_{i l}=0,\quad \hbox{for}~i=1,2,3,\cdots,l, \]
\[a_{1j}+a_{2j}+\cdots+a_{l j}=0,\quad \hbox{for}~j=1,2,3,\cdots,l. \]
Therefore, $-\dfrac{\partial\Delta_pu}{\partial u}$ is symmetrical. If $1<p<2$, we can also conclude that it is symmetrical. In the calculation process of $a_{ij}$, $i,j=1,2,\cdots,l$, for $x\in V$, if there exist some $y$ with $y\sim x$ such that $u(y)=u(x)$, we remove the corresponding term. This means that we only consider the terms that $u(y)\neq u(x)$, for each $y\sim x$. The corresponding Jacobian matrix has a similar form; hence, we omit it here.

\medskip
On the other hand, one knows that its eigenvalues can be written as
\[0=\lambda_{0}<\lambda_{1} \leqslant \lambda_{2} \leqslant \cdots \leqslant \lambda_{l-1},\]
where $l$ is the number of all points in $V$. By the claim above, under the assumption that $\lambda \bar{f}<0$, we may choose a sufficiently small $\varepsilon>0$ such that $G_{\varepsilon}(u, 0)=0$ has a unique solution $u_{\varepsilon}$ satisfying
\begin{equation}\label{deg-1}
2|\lambda| e^{2 u_{\varepsilon}(x)}<\lambda_{1} .
\end{equation}
Substituting $u=u_\varepsilon$ into \eqref{3.12}, we have
\[D G_{\varepsilon}\left(u_{\varepsilon}, 0\right)=-\dfrac{\partial\Delta_pu}{\partial u}\bigg|_{u=u_\varepsilon}+2 \lambda e^{2 u_{\varepsilon}} \mathrm{I},\]
where $I$ represents the $l \times l$ diagonal matrix diag $[1,1, \cdots, 1]$. Moreover, by \eqref{deg-1}, we have
$\lambda_{j}+2 \lambda e^{2 u_{\varepsilon}(x)}>0$, for $j=1,2,\cdots,l-1$. Therefore, by the symmetry of $-\dfrac{\partial\Delta_pu}{\partial u}$, we have
\begin{equation*}\label{3.012}
\begin{array}{ll}\vspace{0.3cm}\displaystyle\operatorname{deg}\left(G_{\varepsilon}(\cdot, 0), B_{R_{\varepsilon}}, 0\right)
&=\operatorname{sgn} \operatorname{det}\left(D G_{\varepsilon}\left(u_{\varepsilon}, 0\right)\right)\\\vspace{0.4cm} \displaystyle
&=\operatorname{sgn}\left\{2 \lambda e^{2 u_{\varepsilon}(x)} \Pi_{j=1}^{l-1}\left(\lambda_{j}+2 \lambda e^{2 u_{\varepsilon}(x)}\right)\right\}\\ \displaystyle
&=\operatorname{sgn} \lambda.
\end{array}
\end{equation*}
Choose $R_0=\max\{R_*,R_\varepsilon^*\}$. Together with \eqref{3.2} and \eqref{3.3}, we have
\[\begin{aligned}
\operatorname{deg}\left(F(\cdot, 1), B_{R}, 0\right) & =\operatorname{deg}\left(F(\cdot, 0), B_{R}, 0\right) \\
& =\operatorname{deg}\left(G_{\varepsilon}(\cdot, 1), B_{R}, 0\right) \\
& =\operatorname{deg}\left(G_{\varepsilon}(\cdot, 0), B_{R}, 0\right) \\
& =\operatorname{sgn} \lambda,
\end{aligned}\]
for arbitrary $R\geqslant R_0$. Thus the proof of Theorem \ref{t1.2} is completed.
\end{proof}

\begin{proof}[Proof of Theorem \ref{t1.3}]
Since $\lambda\bar{f}<0$, then by Theorem \ref{t1.1} and Theorem \ref{t1.2}, we can find a large $R_0>1$ such that for all $R\geqslant R_0$,
\[\operatorname{deg}(F, B_{R}, 0)\neq0.\]
Therefore, by the Kronecker's existence theorem \cite{Chang-2005}, we can conclude that equation \eqref{1.1} has a solution.
\end{proof}

\end{CJK*}

\end{document}